\documentclass{article}
\usepackage{maa-monthly}
\usepackage{mathdots}
\usepackage[capitalise,nameinlink]{cleveref}

\theoremstyle{theorem}
\newtheorem{theorem}{Theorem}

\newtheorem{lemma}[theorem]{Lemma}

\theoremstyle{definition}

\newtheorem*{remark}{Remark}

\begin{document}

\title{On the tower factorization of integers}
\author{Jean-Marie De Koninck and William Verreault}

\maketitle

\begin{abstract}
Under the fundamental theorem of arithmetic, any integer $n>1$ can be uniquely written as a product of prime powers $p^a$; factoring each exponent $a$ as a  product of prime powers $q^b$, and so on, one will obtain what is called the tower factorization of $n$. Here, given an integer $n>1$, we study its height $h(n)$, that is, the number of ``floors'' in its tower factorization. 
In particular, given a fixed integer $k\geq 1$, we provide a formula for the density of the set of integers $n$ with $h(n)=k$. This allows us to estimate the number of floors that a positive integer will have on average. We also show that there exist arbitrarily long sequences of consecutive integers with arbitrarily large heights.
\end{abstract}

\section{Introduction.}

According to the {\it fundamental theorem of arithmetic}, the factorization of an integer $n\ge 2$ into primes is unique apart from the order of its prime factors, a fact which was already known to Euclid (for a modern treatment, see for instance Theorem 1.16 in the book of Niven, Zuckerman, and Montgomery \cite{nzm}). This result is also called the {\it unique factorization theorem}.

Hence, given an integer $n$, one can write it in a unique manner (called the {\it canonical factorization} of $n$) as
\begin{equation} \label{eq:facto1}
n=p_1^{a_1}p_2^{a_2}\cdots p_r^{a_r},
\end{equation}
where $p_1<p_2<\cdots < p_r$ are primes and $a_1,a_2,\ldots,a_r$ are positive integers.
What if we were to write each exponent $a_i$ which is larger than 1 as a product of primes, and then do the same with the new exponents thus generated, and so on\,? We would then obtain a {\it tower factorization}. For instance, the number $n_1=2\,715\,939\,072$ can be written
$$n_1= 2^8 \cdot 3^9 \cdot 7^2 \cdot 11 = 2^{2^3} \cdot 3^{3^2} \cdot 7^2 \cdot 11,$$
the first of these two representations being the canonical factorization of $n_1$, the second one being the tower factorization of $n_1$. This motivates us to introduce the notion of the {\it height}\, $h(n)$ of an integer $n$,   namely the number of floors in the tower factorization of $n$.
Clearly, $h(n)=1$ if and only if $n$ is square-free, $h(p^2)=2$ for each prime $p$, $h(n_1)=3$ and
$\displaystyle{h(65\,536)=h\Big(2^{2^{2^2}}\Big)=4}$. Interestingly, just as with the canonical factorization, one easily sees that the tower factorization of any given integer $n\geq 2$ is unique.
The concept of tower factorization is not new. It was first introduced in 2014 by Devlin and Gnang \cite{dg} as they studied the density of the set $M(q)$ of all integers whose tower factorization contains the prime $q$. They were motivated by computational considerations on shortest encodings appearing in Gnang, Radziwi\l \l, and Sanna \cite{grs}. 

In \cref{kfloors}, for an integer $k\geq 1$ we obtain a formula for the density of the set of integers $n$ with $h(n)\leq k$ and for the density of the set of integers $n$ with $h(n)=k$. 
We first treat the simpler case $k=2$ (resp. $k=3$) in \cref{2floors} (resp. \cref{3floors}). In \cref{aveheight}, we prove that a positive integer has on average approximately $1.4361$ floors, and in \cref{chinese}, we show that there exist arbitrarily long sequences of consecutive integers having arbitrarily large heights. We end this section with open problems.

\section{The general setup.}
Given a positive integer $k$, let $D_k$ (resp. $C_k$) stand for the set of those positive integers $n$ for which $h(n)=k$ (resp. $h(n)\le k$). Also, let
$$d_k = \lim_{N\to \infty} \frac 1N \#\{n\le N: h(n)=k\} \quad \mbox{and} \quad c_k = \lim_{N\to \infty} \frac 1N \#\{n\le N: h(n)\le k\}$$
be the respective densities (as we will see, which exist) of the sets $D_k$ and $C_k$. Clearly,
\begin{equation} \label{eq:ck-dk}
C_k=D_1\cup D_2 \cup \cdots \cup D_k \quad \mbox{and} \quad c_k=d_1+d_2+\cdots + d_k.
\end{equation}
Since the density of square-free integers exists and is equal to $6/\pi^2$ (see for instance Theorem 8.25 in 
\cite{nzm}), we have that
$$c_1=d_1 = \frac 6{\pi^2} \approx 0.6079271019.$$

For each positive integer $k$, we introduce the characteristic function $\lambda_k$ of the set $C_k$, namely the arithmetic function
\begin{equation} \label{eq:lambdak}
\lambda_k(n) := \left\{ \begin{array}{lll} 1 & \mbox{if} & h(n)\le k,\\
                   0  & \mbox{if} & h(n)\ge k+1. \end{array} \right.
\end{equation} 
Clearly $\lambda_1(n)=\mu^2(n)$, where $\mu$ stands for the M\"obius function defined by $\mu(1)=1$, $\mu(n)=0$ if $p^2\mid n$ for some prime $p$ and $(-1)^{\omega(n)}$ otherwise. (Here $\omega(n)$ stands for the number of distinct prime factors of $n$.)

\section{The set of integers having no more than two floors.} \label{2floors}
Note that the tower factorization of an integer has at most two floors if each exponent $a_i$ in its canonical factorization is square-free, so that 
\begin{equation}  \label{eq:lambdamu}
\lambda_2(n)=\lambda_2(p_1^{a_1}p_2^{a_2}\cdots p_r^{a_r})=\prod_{i=1}^{r}\mu^2(a_i).
\end{equation}
Also observe that 
\begin{equation} \label{eq:recur2}
\lambda_2(n)=\prod_{i=1}^r\lambda_1(a_i).
\end{equation} 
The function $\lambda_2(n)$ is multiplicative, that is, $\lambda_2(mn)=\lambda_2(m)\lambda_2(n)$ if $m$ and $n$ are coprime. Indeed, assume that $m=p_1^{a_1}\cdots p_r^{a_r}$ for some primes $p_1<\cdots <p_r$ and positive integers $a_1,\ldots,a_r$ and that
$n= q_1^{b_1}\cdots q_s^{b_s}$ for some primes $q_1<\cdots <q_s$ and positive integers $b_1,\ldots,b_s$, where $\{p_1,\ldots,p_r\}\cap \{q_1,\ldots,q_s\}=\emptyset$. Then, using \eqref{eq:lambdamu},
\begin{align*}
\lambda_2(mn) &=  \lambda_2(p_1^{a_1}\cdots p_r^{a_r} q_1^{b_1}\cdots q_s^{b_s}) \\
&=\mu^2(a_1) \cdots \mu^2(a_r) \mu^2(b_1) \cdots \mu^2(b_s) \\
&=\lambda_2(m)\lambda_2(n),
\end{align*}
thus proving our claim.
 Therefore, using the fact that $\lambda_2(p^r)=\mu^2(r)$ for all primes $p$ and integers $r\ge 1$, we have the Euler product expansion
\begin{align} \label{eq:Euler2}
    \sum_{n=1}^\infty \frac{\lambda_2(n)}{n^s} = \prod_{p}\Big(1+\sum_{r=1}^\infty \frac{\lambda_2(p^r)}{p^{rs}}\Big)
    =\prod_{p}\Big(1+\sum_{r=1}^\infty \frac{\mu^2(r)}{p^{rs}}\Big)\qquad (s>1),
\end{align}
where the infinite product runs over all prime numbers $p$. Letting
$$\zeta(s): = \sum_{n=1}^\infty \frac 1{n^s} = \prod_p \left( 1 - \frac 1{p^s} \right)^{-1} \qquad (s>1)$$
be the Riemann zeta function, \eqref{eq:Euler2} can be rewritten as
\begin{equation} \label{eq:lambda-2}
\sum_{n=1}^\infty \frac{\lambda_2(n)}{n^s} = \zeta(s)\prod_{p}\Big(1-\frac{1}{p^s}\Big)\Big(1+\sum_{r=1}^\infty \frac{\mu^2(r)}{p^{rs}}\Big) =\zeta(s) F(s) \qquad (s>1),
\end{equation}
say.

We will now use Wintner's theorem, which says that if $f$ and $g$ are two arithmetic functions satisfying
$$\sum_{n=1}^\infty \frac{f(n)}{n^s} = \zeta(s) \sum_{n=1}^\infty \frac{g(n)}{n^s} \qquad (s>1)$$
and if $\displaystyle{\sum_{n=1}^\infty \frac{g(n)}{n^s}}$ converges absolutely at $s=1$, then
$$\frac{1}{N} \sum_{n\leq N} f(n) = \sum_{n=1}^\infty \frac{g(n)}{n} +o(1) \qquad (N\to \infty),$$
where $o(1)$ stands for a function of $N$ tending to $0$ as $N\to\infty$.
For a proof of Wintner's theorem, see Theorem 6.13 in the book of De Koninck and Luca \cite{jm-fl}.

An infinite product $\prod_{k=1}^\infty(1+a_k)$ converges absolutely if $\prod_{k=1}^\infty(1+|a_k|)$ converges. It is known (see for instance \cite[Theorem 2.2.8]{prod}) that this happens if and only if $\sum_{k=1}^\infty |a_k|$ converges.
We can show that the infinite product appearing in \eqref{eq:lambda-2} converges absolutely at $s=1$. Indeed, first observe that
$$ 
F(s)  =  \prod_p \Big( 1 - \frac 1{p^s} \Big)\Big( 1 + \frac 1{p^s} +\sum_{r=2}^\infty \frac{\mu^2(r)}{p^{rs}} \Big) 
  =  \prod_p \Big(  1 + \frac{p^s-1}{p^s} \sum_{r=2}^\infty \frac{\mu^2(r)}{p^{rs}} -\frac 1{p^{2s}}  \Big).
$$
Therefore, by the above remark, we only need to show that if we set 
$$a_p :=  \frac{p^s-1}{p^s} \sum_{r=2}^\infty \frac{\mu^2(r)}{p^{rs}} -\frac 1{p^{2s}},$$ 
then $\sum_p |a_p|$ converges.
Since, for any $s\ge 1$, 
$$\sum_{r=2}^\infty \frac{\mu^2(r)}{p^{rs}} \le \sum_{r=2}^\infty \frac 1{p^{rs}} = \frac 1{p^{2s}} \Big( 1+ \frac 1{p^s} +\frac 1{p^{2s}} + \cdots \Big) =\frac 1{p^s(p^s-1)},$$
we get
$$\sum_p |a_p| \le \sum_p \Big( \frac{p^s-1}{p^s} \cdot \frac 1{p^s(p^s-1)} + \frac 1{p^{2s}}\Big) = 
\sum_p  \frac 2{p^{2s}},$$
which converges at $s=1$, thus proving our claim. 

It follows that we may apply Wintner's theorem and conclude that
\begin{equation} \label{eq:lambda2den}
\frac{1}{N}\sum_{n\leq N}\lambda_2(n) = c_2 + o(1) \qquad (N\to\infty),
\end{equation}
where
\begin{eqnarray}  \label{eq:c2}
 c_2 = F(1) & = & \prod_p \Big(1- \frac 1p \Big) \Big( 1 + \sum_{r=1}^\infty \frac{\mu^2(r)}{p^r} \Big) \nonumber \\
  & = & \prod_p \Big( 1 - \frac 1{p^4} + \frac 1{p^5} - \frac 1{p^8} + \frac 1{p^{10}} + \cdots \Big) 
  \approx 0.9559230262.
\end{eqnarray}
Recall that the constant $c_2$ gives the density of the set of those integers $n$ satisfying $h(n)=1$ or $h(n)=2$. We thus also have
\begin{equation} \label{eq:d2} 
     d_2 = c_2 - \frac 6{\pi^2} \approx 0.3479959243.
\end{equation}

\begin{remark}
Estimate \eqref{eq:lambda2den} is not new. In the study of exponentially square-free numbers, Subbarao \cite{sub} obtained \eqref{eq:lambda2den} with an error term of $O(N^{-1/2})$, and Wu \cite{wu} improved this error term unconditionally. T\'{o}th \cite{toth} further improved it under the Riemann hypothesis.   
\end{remark}

\section{The set of integers having no more than three floors.} \label{3floors}
Given $n$ in the form \eqref{eq:facto1}, it is clear that 
$$\lambda_3(n) = \prod_{i=1}^r \lambda_2(a_i).$$

This function is also multiplicative, so the reasoning used in \cref{2floors} still applies. Namely, using the fact that $\lambda_3(p^r)=\lambda_2(r)$, we have, for $s>1$,
\begin{eqnarray*}
\sum_{n=1}^\infty \frac{\lambda_3(n)}{n^s} & = & \prod_p \Big( 1+ \sum_{r=1}^\infty \frac{\lambda_3(p^r)}{p^{rs}} \Big)
= \zeta(s) \prod_p \Big( 1- \frac 1{p^s} \Big) \Big( 1+ \sum_{r=1}^\infty \frac{\lambda_2(r)}{p^{rs}} \Big)\\
& = & \zeta(s) \prod_p \Big( 1 -\frac 1{p^{16s}} + \frac 1{p^{17s}} - \frac 1{p^{48s}} + \frac 1{p^{49s}} - \frac 1{p^{80s}} - \frac 1{p^{81s}} + \frac 1{p^{82s}} + \cdots \Big).
\end{eqnarray*}
Once again Wintner's theorem implies that
$$
\frac{1}{N}\sum_{n\leq N}\lambda_3(n) = c_3 + o(1)\qquad (N\to \infty),
$$
where
\begin{equation} \label{eq:d}
c_3=\prod_{p}\Big(1-\frac{1}{p}\Big)\Big(1+\sum_{r=1}^\infty \frac{\lambda_2(r)}{p^{r}}\Big) \approx 0.9999923551,
\end{equation}
a number which unsurprisingly is in accordance with the fact that very few integers have at least 4 floors.
Combining \eqref{eq:ck-dk}, \eqref{eq:c2} and \eqref{eq:d}, we get
\begin{equation} \label{eq:d3}
d_3= c_3 -d_1 -d_2 = c_3-c_2 \approx 0.0440693289.
\end{equation}

Let $T_k(x)$ stand for the number of positive integers $n\leq N$ with at least $k$ floors. Using a computer, one obtains the following table of values of $T_3(x)$ and $T_4(x)$ for $x=10^r$, $1\le r \le 8$, which supports our previous results.
\vskip 5pt

\begin{center}
\begin{tabular}{|c||c|c|c|c|c|c|c|c|} \hline
$x=10^r$ & 10 & $10^2$ & $10^3$ & $10^4$ & $10^5$ & $10^6$ & $10^7$ & $10^8$ \\ \hline
$T_3(x)$ & 0  & 4      & 43    & 440     & 4408  & 44\,077  & 440\,760 & 4\,407\,699 \\ \hline
$T_4(x)$ & 0  & 0      & 0    & 0     & 1 & 15  & 152 & 1527 \\ \hline
\end{tabular}
\end{center}

\section{The set of integers having no more than $k$ floors.} \label{kfloors}
We now seek a formula for $c_k$ and $d_k$ for a general $k\geq 2$. Since it is clear from  \eqref{eq:ck-dk} that
$$
d_k=c_k-c_{k-1}\qquad (k\geq 2),
$$
it suffices to obtain a formula for $c_k$.
Recalling the definition of $\lambda_k$ given in \eqref{eq:lambdak},
it is easy to see
that \eqref{eq:recur2} can be generalized to
$$
 \lambda_{k}(n)=\prod_{i=1}^r\lambda_{k-1}(a_i) \qquad (k\geq 2).
$$

Since each $\lambda_k(n)$ is multiplicative, it follows as before that
\begin{align*}
    \sum_{n=1}^\infty \frac{\lambda_k(n)}{n^s} = \zeta(s)\prod_{p}\Big(1-\frac{1}{p^s}\Big)\Big(1+\sum_{r=1}^\infty \frac{\lambda_{k-1}(r)}{p^{rs}}\Big)=\zeta(s)F_k(s) \qquad (s>1),
\end{align*}
say, and that $F_k(s)$ converges at $s=1$. Wintner's theorem then gives
$$
\frac{1}{N}\sum_{n\leq N}\lambda_k(n) = c_{k} + o(1)\qquad (N\to \infty),
$$
where
$$
c_{k}=\prod_{p}\Big(1-\frac{1}{p}\Big)\Big(1+\sum_{r=1}^\infty \frac{\lambda_{k-1}(r)}{p^{r}}\Big) \qquad (k\geq 2).
$$
By definition of $\lambda_k(n)$, this is precisely the density of the set of integers $n$ with ${h(n)\leq k}$.

\section{Average height.} \label{aveheight}
If we pick a positive integer at random, how many floors should one expect its tower factorization to have? This average height depends on the various densities $d_k$. Since a positive integer can have any number of floors, one can prove that
\begin{equation} \label{eq:aveh}
\lim_{N\to\infty}\frac{1}{N}\sum_{n\leq N}h(n)
=\sum_{k=1}^\infty k d_k.
\end{equation}

In order to establish that \eqref{eq:aveh} holds, we need to obtain upper bounds for the densities $d_k$ which will ensure that the series $\displaystyle{\sum_{k=1}^\infty kd_k}$ does indeed converge. To do so, we proceed as follows.

For convenience, given an integer $k\ge 2$, we set $2^{(k)}:=2^{\iddots^2}$, the minimal tower with $k$ floors, so that in particular $2^{(2)}=2^2$ and $2^{(3)}=2^{2^2}$.   We begin by establishing two important inequalities.

\begin{lemma}\label{lem:1}
(i) For any given integer $a\ge 2$, we have that $\displaystyle{\sum_p \frac 1{p^a} <  \frac 2{2^a}}$.\\
(ii) For any given integer $k\ge 2$, we have that $\displaystyle{d_k< \frac 4{2^{(k)}}}$.
\end{lemma}

\begin{proof}
It is well-known that $\displaystyle{\sum_p \frac 1{p^2} =0.452\ldots}$ (see for instance page 95 in the book of Finch \cite{Finch}), thus proving $(i)$ in the case of $a=2$. Hence, let $a\ge 3$ and first observe that it is clear that
$$
\sum_p \frac 1{p^a} < \zeta(a)-1.
$$
In light of this last inequality, the proof of part $(i)$ will be complete if we can prove that
\begin{equation} \label{eq:102}
\zeta(a) -1 < \frac 1{2^{a-1}} \quad \mbox{for all }a\ge 3,
\end{equation}
or equivalently that
\begin{equation} \label{eq:103}
\frac 1{3^a} + \frac 1{4^a} + \frac 1{5^a} + \cdots  < \frac 1{2^a} \quad \mbox{for all }a\ge 3.
\end{equation}
It is immediate that for any integer $a\ge 3$,
$$\frac 1{3^a} + \frac 1{4^a} + \frac 1{5^a} + \cdots < \int_2^\infty \frac{dt}{t^a} = \frac 1{(a-1)2^{a-1}}\le \frac 1{2^a},$$
thus proving (\ref{eq:103}), from which (\ref{eq:102}) follows, thereby completing the proof of part $(i)$.

To prove part $(ii)$, we first note that the case $k=2$ is immediate because \eqref{eq:d2} implies that $d_2<1$. For the general case $k\ge 3$, observe that any given $n\in D_k$ (that is, an integer $n$ with $h(n)=k$) can be written in a unique way as $n=p^a\cdot m$, where $p$ is the smallest prime $p$ dividing $n$ for which $p^a\mid\mid n$ with $h(a)=k-1$ (implying that $a\ge 2^{(k-1)}$). We then have, using part $(i)$,
\begin{eqnarray*}
\sum_{n\le N \atop h(n)=k} 1 & \leq & \underset{\substack{p^am\leq N \\ h(a)=k-1}}{\sum_{p}\sum_{m\geq 1}}1
 \leq  N \sum_{a\ge 2^{(k-1)}} \sum_p \frac 1{p^a} < N \sum_{a\ge 2^{(k-1)}}\frac 2{2^a} = N \frac 4{2^{(k)}},
\end{eqnarray*}
thus completing the proof of part $(ii)$.
\end{proof}

Observe that since, trivially, $2^{(k)}\ge 2^{k}$ for all $k\ge 2$, it follows from part $(ii)$ of Lemma \ref{lem:1} that
$$\sum_{k=1}^\infty kd_k = d_1 + \sum_{k=2}^\infty kd_k < 1 +  4 \sum_{k=2}^\infty \frac k{2^{k}}  =1 + 4 \cdot \frac 32  = 7,$$
implying that the series $\displaystyle{ \sum_{k=1}^\infty kd_k }$ does indeed converge.

So, given a large integer $N$ and letting $K=K_N$ be the smallest positive integer for which $h(n)\le K$ for all $n\le N$, we then have 
$$ \frac 1N \sum_{n\le N} h(n)  =  \frac 1N \sum_{k=1}^{K_N} k \sum_{n\le N \atop h(n)=k} 1= \sum_{k=1}^{K_N} k \frac 1N \sum_{n\le N \atop h(n)=k} 1,$$
so that 
$$\lim_{N\to \infty} \frac 1N \sum_{n\le N} h(n) = \sum_{k=1}^\infty kd_k,$$
which proves \eqref{eq:aveh}.

Denoting by $\mathfrak{a}$ the average value of the number of floors of a positive integer, one can round it to four decimal places by only using the values of $d_1$, $d_2$, and $d_3$ obtained above (the other $d_i$'s having no incidence on these four decimals). One gets that $\mathfrak{a} \approx 1.4361$.

\section{Consecutive integers with arbitrarily large heights and various open problems.} \label{chinese}
Tower factorizations raise many interesting questions.
For instance, using a computer one can check
that the first occurrence of three consecutive integers each with height at least 3 is provided by the triplet $n_0$, ${n_0+1}$, $n_0+2$, where
$$n_0=248\,750=2\cdot 5^{2^2} \cdot 199,\quad n_0+1=3^{2^2} \cdot 37 \cdot 83,\quad  n_0+2=2^{2^2}\cdot 7 \cdot 2221.$$
What about four consecutive integers with the same property\,? What about greater heights\,? Perhaps surprisingly, all are possible. Indeed, one can show that given arbitrary integers $\ell\ge 2$ and $k\ge 2$, there exist infinitely many integers $n$ with the property that $h(n+j)\ge k$ for $j=0,1,\ldots,\ell-1$. Indeed, consider the system of the $\ell$ congruences
\begin{eqnarray*}
n & \equiv & 0 \pmod{ 2^{2^{\iddots^2}} },\\  
n+1 & \equiv & 0 \pmod{ 3^{2^{\iddots^2}} },\\
 & \vdots &  \\
n+\ell-1 & \equiv & 0 \pmod{ P_{\ell}^{2^{\iddots^2}} },
\end{eqnarray*}
where $P_{\ell}$ is the $\ell$-th prime and each of the numbers $\displaystyle{p_i^{2^{\iddots^2}}}$ represents a tower of height $k$. Then,
set  $$Q_\ell:=\displaystyle{2^{2^{\iddots^2}}\cdot 3^{2^{\iddots^2}} \cdots \; P_{\ell}^{2^{\iddots^2}}}.$$
 By the Chinese remainder theorem one can claim the existence of a number $n_0< Q_\ell$ for which $h(n_0+j)\ge k$ for ${j=0,1,\ldots,\ell-1}$. Actually, more is true, namely the existence of infinitely many integers $n$ satisfying
$h(n+j)\ge k$ for $j=0,1,\ldots,\ell-1$ by simply considering the numbers $n=n_0+r\cdot Q_\ell$ with $r=0,1,2,\ldots$, thus proving our claim.

Finally, it would be nice to obtain explicit expressions for the densities $c_k$ and $d_k$, $k\geq 3$, instead of the recursive expressions obtained here. It is also an open problem to find an explicit expression for the density of the set $M(q)$ of all integers whose tower factorization contains the prime $q$, which we alluded to in the Introduction. For instance, when $q=2$, we certainly have that this density is at least $0.5$, but in fact it is $\approx 0.577 350$. We refer the reader to \cite{dg} for more information on this problem.

\begin{acknowledgment}{Acknowledgment.}
The authors wish to thank the referees and the editor for their very helpful comments.
\end{acknowledgment}

\begin{biog}
\item[Jean-Marie De Koninck] obtained his Ph.D. in mathematics from Temple University in 1972 under the supervision of Emil Grosswald. After 44 years as a faculty member at Université Laval in Quebec, Canada, he retired in 2016 and is now Professor Emeritus. His main research interest is the multiplicative structure of integers.  He is still involved in math outreach and takes pleasure in swimming and writing books. His latest, co-authored with Nicolas Doyon, is \textit{The Life of Primes in 37 Episodes}.

\begin{affil}
D\'epartement de math\'ematiques et de statistique, Universit\'e Laval, Qu\'ebec G1V 0A6, Canada\\
jmdk@mat.ulaval.ca
\end{affil}

\item[William Verreault] is a master’s student at Université Laval and a soon-to-be Vanier Scholar at the University of Toronto. While his main research interests lie in analysis, he enjoys working on various number theory problems.

\begin{affil}
D\'epartement de math\'ematiques et de statistique, Universit\'e Laval, Qu\'ebec G1V 0A6, Canada\\
william.verreault.2@ulaval.ca
\end{affil}
\end{biog}
\vfill\eject

\end{document}